\documentclass[11pt]{amsart}
\usepackage{epsfig,amsmath,amssymb,amsthm,bm}
\usepackage{enumerate}
\usepackage{amsfonts,amsmath}

\newcommand{\IZ}{\mathbb{Z}}                       
\newcommand{\IN}{\mathbb{N}}  
\newcommand{\IR}{\mathbb{R}}   

\newtheorem{stat}{Statement}

\newtheorem{prop}[stat]{Proposition}
\newtheorem{cor}[stat]{Corollary}

\newtheorem{lemma}[stat]{Lemma}

\theoremstyle{remark}

\renewcommand{\epsilon}{\varepsilon}
\renewcommand{\phi}{\varphi}

\usepackage[usenames]{color}

\parskip        \smallskipamount

\usepackage{marginnote}

\begin{document}

\title[A construction of the Stable Web]{A construction of the Stable Web}

\author{Thomas Mountford}
\address{Thomas Mountford:
Ecole Polytechnique F\'ed\'erale de Lausanne,
Institut de Math\'ematiques,
Station 8, 1015 Lausanne, Switzerland}
\email{thomas.mountford@epfl.ch}

\author{Krishnamurthi Ravishankar}
\thanks{KR was supported by Simons Collaboration Grant 281207}
\address{Krishnamurthi Ravishankar:
NYU -ECNU Institute of Mathematical Sciences at  NYU
Shanghai, 3663 Zhongshan Road North, Shanghai, 200062 China}

\author{Glauco Valle}
\address{Glauco Valle:
Instituto de Matem\'atica,
Universidade Federal do Rio de Janeiro}

\begin{abstract}
We provide a process on the space of coalescing cadlag stable paths and show convergence in the appropriate topology 
for coalescing stable random walks on the integer lattice.


\end{abstract}

\maketitle
\section{\Large\bf Introduction}
A system of coalescing Brownian motions starting at ``every"  point in $\mathbb{R}$ and evolving independently before coalescence was first introduced by Arratia in \cite{A1,A2}. This system have being studied by several authors and motivated the question about the existence of a system of coalescing Brownian Motion starting at ``every"  point in the space-time plane $\mathbb{R}^2$. Such an object is called the Brownian Web and was introduced by Fontes, Isopi, Newman and Ravishankar in \cite{FINR}. In the same paper they prove weak convergence to the Brownian Web under diffusive scaling of the system of simple symmetric one-dimensional coalescing random walks starting on each point in the space-time lattice $\mathbb{Z}^2$. Later Newman, Ravishankar and Sun \cite{NRS} proved an invariance principle related to the Brownian Web, they stablished the convergence to Brownian web  for systems of one-dimensional coalescing random walks under finite absolute fifth moment of the transition probability (allowing for crossing of paths unlike the nearest neighbor walks). 

More recently, Evans, Morris and Sen [EMS] studied a system of coalescing $\alpha$-stable processes, $\alpha > 1$, starting at every point in $\mathbb{R}$. As Arratia \cite{A2} did for Brownian Motion, they proved that the system of $\alpha$-stable processes are locally finite for every time $t>0$. And based on this, our main motivation here is to build a stable version of the Brownian Web or simply a "Stable Web" and also prove an invariance principle for it.  

In this note we make a first step at defining the stable web. Together with Hao Xue we will generalize the domain of applicability in a subsequent work and show that the object defined is equivalent to an object with a more general "smoother" topology.  Our objective here is simply to define a reasonable metric on the space of collections of cadlag paths that gives convergence to the "stable web" for suitably normalized coalescing random walks.

\section{\Large\bf The Age Process}
\label{sec:agedproc}

In the following we have for $ n \in \IZ_+  \ D_n \ = \ \IZ/2^n $.  

\noindent We follow  \cite{EMS} and consider systems of coalescing identically distributed stable processes $\overline{X}^{D_n} = \{X^{n,x} = (X_t^{n,x})_{t\geq 0}: x \in D_n\}$ with stable index $\alpha \in (1,2)$ such that $X_{0}^{n,x} = x = i 2^{-n}$ for integers $i$ and $n$. The stable processes evolve independently until coalescence. The rules of precedence will be arbitrary for points in $D_{n} / D_{n-1}$ but lower order points will have coalescence precedence so that $\overline{X}^{D_n} \subset \overline{X}^{D_{n+1}}$ for each $n \ge 1$. For $x \in D = \cup_{n \ge 1} D_{n}$, take $n$ such that $x \in D_n / D_{n-1}$ and simply write $X^{x,n} = X^x$. Moreover we denote  $\overline{X} \ = \ \cup_{n \ge 1} \overline{X}^{D_n}$   and for every $t>0$, $\overline{X}^{D_n}_{t} \ =  \ \cup _{x \in D_n} X_t ^{x}$ and $\overline{X}_t \ =  \ \cup _{x \in D} X_t ^{x}$ which are the time level sets associated to the set valued processes $\overline{X}^{D_n}$ and $\overline{X}$ respectively. Arguing as in \cite{BG} we have:

\begin{prop}
\noindent There exists  $K < \infty$ so that $\hspace{0.2cm} \forall \hspace{0.2cm} n \ge 1$ the density of the process, $ \overline{X}^{D_n}$ at time $t < \infty \hspace{0.2cm} D (n,t)$
given by the a.s. limit of   
$$
 \lim_{M\rightarrow \infty} \frac{1}{2M} \# \big(\overline{X}^{D_n}_{t} \cap [-M,M] \big) 
$$
satisfies $D (n,t) \leq \frac{K}{t^{\frac{1}{\alpha}}}$.
\end{prop}

\vspace{0.3cm}

\noindent It follows from the fact that $0$ is regular for the stable processes that if we choose a $x \notin D$ and start a stable process at $x$ at time $0$ which will coalesce with processes starting at $D$ (given precedence to the latter) then $P (X^{x}_{t} \in \overline{X}_{t} \ \forall \, t>0) = 1$.  For any $x \in \mathbb{R}$ we can indeed unambiguously define a stable process $(X^{x}_{t})_{t \geq 0}$ which has the same distribution as a stable process starting at $x$ and such that  $t> 0$, $X^{x}_{t} \in \overline{X}_{t}$. 
So we can think of the above process as a collection of coalescing stable processes starting ``on" $ \mathbb{R} $
\vspace{0.2cm}
\noindent For any set $E \subset \mathbb{R}$ we denote $\overline{X}^E = \{ X^x : x \in E\}$. It follows that if for any nested collection of ``translation invariant" points $V_{n}$ with $V \ = \ \cup_{n \geq 1} V_{n}$ dense we have (with the coalescence rules with $\overline{X}_{t}$ as above) that $\overline{X}^{V_{n}}_{t} \subset \overline{X}_{t} \hspace{0.2cm} \forall \hspace{0.2cm}  t > 0$  with probability 1. But equally we can show a.s that $\overline{X}_{t} \subset \overline{X}^{V}_{t}$. Furthermore for any strictly positive $c$ we can choose $V_{0}$ to be points spaced $c$ apart and containing $0$ and $V_{n}$ obtained from $V_{n-1} $ by adding the midpoints between neighbours. Then we have that 
$$
\overline{X}^{V_{n}}_{t} \stackrel{D}{\longrightarrow} \overline{X}_{t} \, , \quad 
\overline{X}^{D_{n}}_{t} \stackrel{D}{\longrightarrow} \overline{X}^V_{t} \quad \textrm{and} \quad 
\frac{1}{c} \hspace{0.2cm} \overline{X}^{V_{n}}_{c^{\alpha} t} \stackrel{D}{=} \overline{X}^{n}_{t} \, ,
$$
together this yields.

\begin{prop} \label{density}
\noindent For the process $\overline{X}_{t}$ the density at time $t$ is equal to $k/t^{\frac{1}{\alpha}}$ for some $k$ depending on our choice of the stable process.
\end{prop}

\vspace{0.3cm}

To construct the stable web we now consider coalescing stable processes starting at \emph{times} $t \in D$. 

\smallskip

The first step is for $D_{0} = \IZ^{1}$. We define the stable coalescing processes starting at $\{i\} \times \IR$.
 For $t \in (i, i + 1]$ let this process be $\overline{X}^{i}_{t}$. At time $t = i + 1$, $\overline{X}^{i}_{i + 1}$ will be a countable collection of points on $\{i +1\} \times \IR.$ As such they can be continued on interval $[i+ 1, i + 2]$ so that $\overline{X}^{i}_{s} \subset \overline{X}^{i + 1}_{s}$ for every $s \in [i+ 1, i + 2]$  with probability 1. Continuing we have $\forall \hspace{0.2cm} i < j \hspace{0.2cm} \overline{X}^{i}_{s} \subset \overline{X}^{j}_{s} \hspace{0.2cm} \forall \hspace{0.2cm} s > j$ a.s.

\smallskip

We now proceed in an analagous manner adding in stable coalescing processes at times $D_{n} / D_{n-1}$ to obtain a collection of processes $\{\overline{X}^{d}_{t}\}_{t > d}$ for $d \in D$ with the property that a.s. for every $d < d'$, $t > d'$ we have that $\overline{X}^{d}_{t} \subset \overline{X}^{d'}_{t}$. We use the notation $X^{d,d'}=(X^{d,d'}_t )_{t \geq d} $ to denote the stable process beginning at (dyadic) time $d$ at spatial (dyadic) point $d'$. 


We now define the \emph{age} of a process (or path) $(\gamma(s))_{s> d}$ of  $\overline{X}^{d}$: the age of $(\gamma(s), s)$ is simply $s \ - \ \inf (d' < d: \gamma(s) \in \overline{X}^{d'}_{s}) .$  So the age of $ ( \gamma (s),s) $ increases continuously  at rate 1 but then jumps when the path coalesces with an older path.  
\smallskip 

From this we can for each $\epsilon$ consider the $\epsilon$-processes or $\epsilon$-paths consisting of paths whose age is greater than or equal to $\epsilon$: For every $d, d' \in D $, for the path with $Y=X^{d,d'}_d  = d' $, we let $A_Y ( \delta ) = \inf \{s > d: (X^{d,d'}_s,s) \mbox{  has age } \geq \delta\}
$, then the path $(Y_t)_{t \geq d} $ is replaced by $(Y_t)_{t \geq A_Y( \epsilon)}$ which we call an $\delta$-process.  Since we will be interested in distinct $\delta$-paths, we can regard paths for which the age at $A( \delta )$ strictly exceeds $\delta$ (i.e. the age is raised as the result of a coalescence with an ``older" path)  as being suppressed. Thus in a certain way we are imposing that older paths have coalescence precedence over the others. We leave the reader to check that there is no infinite regress and not all $\delta$-paths are suppressed and the union of all points touched by unsuppressed $\delta$-paths at a time $t$ gives all the points of age greater than $\delta$ at time $t$.   We similarly note that almost surely no path can achieve age precisely $\delta$ via a coalescence.  We write $\Phi_\delta $ for this operation on the set of aged paths.  

In general we say a collection of aged paths is a set of triples $(I, \gamma, a)$ (where $I$ is an open interval, $\gamma: I \to \IR $ is a cadlag function and
$a: I \to \IR _+$ is a cadlag function satisfying for $s' < s$, $a(s) \geq a(s') + (s-s')$. When $I = (b,\infty)$ we simply write $(I, \gamma, a) = (b, \gamma, a)$. 

We put  $\mathcal{X} = \cup_{d \in D} \{ ([d,\infty), Y,a_Y): Y \in \overline{X}^{d}\}$  and $\mathcal{X}_\delta = \Phi_\delta (\mathcal{X})$ as our systems of coalescing stable aged processes and coalescing $\delta$-processes. Given a countable dense collection of space time points $\{(x_i, t_i ) \}_{i=1} ^ \infty $ we can also define for  stable aged processes beginning at these space time points which are subordinate (or contained in)  to our "stable web" aged process. 

In our system of paths, as already noted, a $\delta$-path $\gamma : [c, b] \rightarrow \mathbb{R}$ is a path so that  $(\gamma(c), c)$ has age at least $\delta$. If $(\gamma (c), c)$ has age exactly $\delta$ we say it is $\delta$-fixed path.

\smallskip

Finally note that Proposition \ref{density} yields the following corollary:
\begin{cor} \label{denage}
The density of points of $\mathcal{X}$ with age in the interval $(a, a+ \epsilon) $ at a given time $t$ is equal to $k/ a^{1/ \alpha} \ - \ k/ (a+ \epsilon ) ^{1/ \alpha} $.
\end{cor}

\medskip

We will also need the following result which follows from the fact that the density of processes of age at least $M$ at a particular time tends to zero as $M$  becomes large.

\smallskip

\begin{lemma} \label{nobigage}
Given $\epsilon > 0 $ and $N < \infty$, there exists $M = M(\epsilon,N) < \infty$ so that outside probability $\epsilon $ every path of $\mathcal{X}$ that intersects space time square $[-N,N]^2$ has age less than $M$ at time $N$. 
\end{lemma}
\begin{proof}
This simply follows from the fact that by Proposition \ref{density}, the density of coalescing processes, started at time $-(N+M)$ has density $k/(2N+M)^{1/ \alpha}$ at time $N$. So the chance that one such process is in spatial interval $[-N,N]$ at time $N$ is less than $2kN/(2N+M)^{1/ \alpha}$.   Let $c(N) > 0$ be the infimum of  the conditional probability  a path be in $[-N,N]$ at time $N$  given that it hits $[-N,N]^2$  . The probability that the event of interest occurs is bounded above by $2kN/c(N)(2N+M)^{1/ \alpha}$ which will be less than $\epsilon$ for large enough $M$ depending on $N$ and $\epsilon$.
\end{proof}

\smallskip

\smallskip

\section{\Large\bf Topology}
First recall (see e.g. \cite{EK}) the definition of the $\delta$-modulus of continuity for a cadlag path $\gamma:[c,d] \rightarrow \mathbb{R}$:
$$
\omega ( \delta, \gamma, [e,f]) \ = \ 
\inf _{ t_i - t_{i-1} \geq \delta} \sup _ i \sup _{s,t \in  [ t_{i-1}, t_i )} | \gamma (t) - \gamma (s)|
$$
for $[e,f]  \subset [c,d]$.  This quantity is important  for determining the compactness of  sets of cadlag paths.

We use the metric $d_1$ between two cadlag paths $\gamma_{1} : [a, b] \rightarrow \IR$ and $\gamma_{2} : [c, d] \rightarrow \IR$  (typically but not always we will have $b \ = \ d \ = \ \infty$) where  
\begin{eqnarray*}
d_1 (\gamma_{1}, \gamma_{2}) & = &|tanh(a) - tanh(c) | +  \\ 
& & \!\!\!\!\!\!\!\!\!\!\!\!\! \inf _{ g:[a,b] \to [c, d] } \left[   \sup_{a \leq t \leq b} e^{-|t|}|( ( \gamma_1  (t), t) -( \gamma_{2} (g (t)), g(t)) | \wedge 1 ) + \sup_{a < s < b} e^{-|s|}|g'(s)-1| \right]  \, ,
\end{eqnarray*}
where  the infimum is over continuous piecewise differentiable  bijections $g$. This amounts to a compactification of space-time as in \cite{FINR}.  When dealing with finite space time rectangles $[S,T] \times [A,B] $, we will use the equivalent metric
\begin{eqnarray*}
d (\gamma_{1}, \gamma_{2}) & = & |a - c | + \\
& & \inf_{g:[a,b] \to [c, d]} \left[   \sup_{a \leq t \leq b} ( ( \gamma_1  (t), t) -( \gamma_{2} (g (t)), g(t)) | \wedge 1 ) + \sup_{a < s < b} |g'(s)-1| \right] \, .
\end{eqnarray*}

\smallskip

In dealing with aged paths defined over finite intervals 
$$
(\gamma_{i}(t), a_i(t)) : [c_i, b_i] \rightarrow \IR \times (0, \infty )
$$
for $i = 1,2$, we simply take 
$$
d( (\gamma_1, a_1), (\gamma_2, a_2) ) \ = \ d( \gamma_1, \gamma_2 )
\vee d ( a_1,a_2 ) 
$$
and similarly for $d_1$.
\vspace{0.1cm}
In the following, when speaking of distance between aged paths $( \gamma_1, a_1), (\gamma_2, a_2)$, we will abuse notation and write $d_1( \gamma_1, \gamma_2 )$.
It follows immediately,

\medskip

\begin{lemma} For a cadlag function $f:[0, T] \to \IR$ let $f^{\eta}$ be its restriction to $[\eta, T]$ (for $\eta > 0$). 
$\forall \sigma > 0$, there exists $\eta_{0}$ so that $d(f, f^{\eta}) < \sigma$ for every $0 \leq \eta \leq \eta_{0}.$
\end{lemma}
\begin{proof}
Take $h $ to be such that $\sup_{s \leq h} |f(s)-f(0)| \ \leq \ \sigma/10.$  Now (for $h\ > \  \eta \  > \ 0 $) define path $g: \ [0,T] \rightarrow [\eta, T $] by
$$
g(s) = \ s \ \mbox{ for } s \in [h,T] ; \quad g( . ) \ \mbox{ is linear bijection } [0,h] \ \rightarrow \ [\eta ,h].
$$
Then this shows that $ d(f, f^{\eta}) < 2 \sigma / 10 + \ \eta/h \ + \ \eta $ and so the result follows.
\end{proof}

\medskip

The above argument in fact yields 
\begin{lemma}
For any $ h > \eta $,
$$
d(f^{\eta}, f) \leq \eta \ + \ \frac{\eta }{h} + 2 \sup_{0 \leq s \leq h} \vert f(0) - f (s) \vert \, .
$$
If we are dealing with aged paths
$$
d((f^{\eta},a^ \eta ) , (f,a)) \leq \eta \ + \ \frac{\eta }{h} + 2 ( \sup_{0 \leq s \leq h} \vert f(0) - f (s) \vert \vee a(h)-a(0)) \, .
$$
\end{lemma}

\medskip

\begin{cor} For a stable process $(X(t) : t \geq 0)$, $\forall \sigma \ > \ 0$ there exists $\eta_{0} >0$ so that $\forall \ T>0$,
$$
P ( \forall \ \eta \leq \eta_{0} ,d(X, X^{\eta})< \sigma)   > 1 - \frac{\sigma^{2}}{10^{6}} \, ,
$$ 
where we consider the restrictions of $X$ and $X ^ \eta $ to $[0,T]$ and $[\eta, T]$ respectively.  
\end{cor}
 Similarly, if for process $X$ and time interval $I$ within its domain of definition, we write $X^I $ as the process with time  restricted to $I$, then we have

\begin{cor} \label{corstab} For a stable process $(X(t) : t \geq 0)$, $\forall \sigma \ > \ 0, \ \epsilon > 0$ there exists $\eta_{0} >0$ so that $\forall \ t> \ \epsilon$ and
$0 \leq \eta_1, \eta_2 \ \leq \eta_0 $,
$$
d(X^{[0,t]} , X^{[\eta_1,t + \eta_2]} ) <  \sigma 
$$
outside probability  $\frac{\sigma^{2}}{10^{6}}$.
\end{cor}
\medskip

We will examine the systems of aged paths $\mathcal{X}$ and $\mathcal{X}_\delta$ considering them as random elements of a proper path space which we now define in a natural way based on our previous considerations. Let $G$ be the space of aged paths: $\{(b,\gamma,a)\} $ where  $b \in \mathbb{R}$,  $\gamma$ and $a$ are cadlag functions defined on $(b, \infty)$ and $a$ satisfies $a(s) \geq a(s') + (s-s')$ for every $s>s' >b$. For each time-space rectangle $[S,T]\times [A,B]$, we define the ``restriction" of an aged path $\{ (b, \gamma, a )  \}$ to $[S,T]\times [A,B]$ as the triple
$$
\{ (b^{[S,T]\times [A,B]}, \gamma ^ {[S,T]\times [A,B]}, a ^ {[S,T]\times [A,B]})  \} 
$$ 
where:\\
$$
b^{[S,T]\times [A,B]} \ = \ \inf\{r>b \vee S: \gamma(r) \in [A,B]  \}
$$ 
and if $b> T$, 
$$
( \gamma ^ {[S,T] \times [A,B]}, a ^ {[S,T]\times [A,B]}) \  = \  \emptyset \, ,
$$
otherwise, for $r>b^{[S,T]\times [A,B]}$, 
$$
\gamma ^ {[S,T]\times [A,B]} (r) \ = \ ( \gamma(r) \vee A) \wedge B 
$$ 
and define $a^ {[S,T]\times [A,B]} (.)$ analogously.

For $S = A = -N, \ B = T = N$, we denote by $\Psi_N $ the map that associates to a path in $G$ its restriction.  We note that this is not a continuous operator for the given metric between paths.

We have the semimetric $\rho_{[S,T]\times[A,B]}$  on $G$ defined by the d-distance between $(b, \gamma, a ) $ and $(b', \gamma ', a' )$ restricted to 
interval $[S,T]\times[A,B]$.  

We denote the composition $ \Psi_N \circ \Phi_{2^{-N}} $ by $\Pi_N$ (recall $\Phi_\delta$ is defined at the end of Section \ref{sec:agedproc}).
We now consider the metric between aged paths (which by abuse of notation we also denote as $\rho$)  by 
$$
\rho ( (b, \gamma, a), (b', \gamma', a') )  \ = \  \sum_{N=1}^\infty 2^{-N} \min \big\{ 1,  \rho_{[-N,N]^ 2} \big( \Pi_N(b, \gamma, a), \Pi_N(b', \gamma', a') \big) \big\}  \, .
$$
 (Here for each $N$ in evaluating $\rho_{[-N,N]^ 2}$, $ \emptyset $ is regarded as a path at distance $1$ from each distinct other path).

The metric is artificial in that it privileges certain cutoffs and rectangles, whereas the spatial or temporal integer values are not special for stable processes and the ages $2^{-N}$ are not significant for our coalescing system. However this is a positive in our approach as we will be able to argue that the lack of continuity of our projections $\Pi_N $ is ultimately not a problem.

We, as usual, take $H$ to be the set of closed subsets of $G $ with the Hausdorff metric between closed sets based on the just defined metric $ \rho $ between paths.
Extending the abuse of notation, we still use $\rho $ to denote this metric.
The law of the stable web will be a probability measure on open (Borel) subsets of $H$ with respect to this metric.
 We have the usual criterion for tightness (see \cite{FINR}).



For every $N\ge 1$ fix $\epsilon_N > 0$, $M_N > 0$ and $\delta_N \in [0,1)^\IN $ a sequence tending to zero and put $\vartheta = (\epsilon_N, M_N, \delta_N)_{N \geq 1}$. We denote by $K(\vartheta)$ the set of collections of aged paths in $G$ such that for each collection $C \in K(\vartheta)$ and for each integer $N \geq 1$:
\begin{enumerate}
\item[(i)] the number of paths in $\Pi_N (C) $ is less than $M_N$;
\item[(ii)] the age of every path in $\Pi_N (C) $ is less than $M_N$ throughout; 
\item[(iii)]  every path in $\Pi_N (C) $ is contained in $[- N, N] \times [- M_N, M_N]$;
\item[(iv)]  every path, $(b,\gamma,a) \in \Pi_N (C) $ has $\omega ( 2^{-r}, \gamma, [T_N,N]) \leq \delta_N (r)$ for every $r \in \mathbb{N}$, where $T_N = b^{[-N,N]^2}$; 
\item[(v)]  every path, $(b,\gamma,a) \in \Pi_{N+1} (C) $ has $\gamma (T_N) \in [- N+ \epsilon _N, N-  \epsilon _N]$
and prior to this it did not enter $ [- N- \epsilon _N, N+  \epsilon _N]$;
\item[(vi)]  the age process of every path $\gamma \in \Pi_{N+1} (C)$ makes no jump while the age has value in $[2^{-N}
- \epsilon_N,2^{-N}
+ \epsilon_N ]$; 
\item[(vii)] the age process of every path $\gamma \in \Pi_N (C)$ makes no jumps within time $\epsilon_N$ of each other;
\item[(viii)] the age process of every path $\gamma \in \Pi_{N+1} (C)$ does not have age in interval $(2^{-N} - \epsilon_N, 2^{-N} + \epsilon_N)$
at times in   $[-N- \epsilon_N , -N + \epsilon_N]$  and $[N- \epsilon_N, N+ \epsilon_N]$.
\item[(ix)]  every path, $(b,\gamma,a) \in \Pi_{N+1} (C) $ has $d( \gamma^{ [T_N,N]}, \gamma^{ [T_N+ \eta_1,N+ \eta_2 ]}    ) \leq 2^{-r} $ for every $r $ with $0 \leq \delta_N(r) < \epsilon_N/2 $
and $0 \leq \eta_1, \eta_2 \leq \delta_N (r)$, where again $T_N = b^{[-N,N]^2}$; 
\end{enumerate}
 
\medskip

\begin{prop} \label{prok}
The sets $K(\vartheta)$ are compact.
%
\end{prop}

\begin{proof}
Given a sequence of collections $C_1, C_2, \dots $ in $K$, we prove that we can find a convergent subsequence which converges to an element of $K$.

If we fix $N$ and consider the paths $ \gamma $ that are in $\Pi_N (C_n) $ for some $n$, then by
conditions (iii) and (iv) this set is compact.  Similarly (ii) and (vii) ensure that the age processes will be compact.  Since the set of compact sets of trajectories on a bounded domain endowed with the Hausdorff metric is compact, that is  $\{\Pi_N(C) | C \in K\}$ is compact  for each $N$, we can take a subsequence, $(C_{n^N_j})$ of $(C_n)$  so that 
$(\Pi_N (C_{n^N_j}) )$ converges.    By Cantor diagonal method, we arrive at a subsequence
$C_{n_i}, i \geq 1$, such that for every $N$ the sequence $(\Pi_N (C_{n_i}))$ converges to some collection
$D_N$.  We must show that from this collection $(D_N)_{N \geq 1}$, we can find an aged path collection $C$ so that $C_{n_i}$ converges to $C$.  It will be clear that any such limit is in $K(\vartheta)$ so the principal task is to produce aged path collection$C$ so that for every $N$, $\Pi_N(C) = D_N$.  The essential step is to show that 
for each $N$, $\Pi_N(D_{N+1}) = D_N $.  Since $\Pi_N(C_{n_i}) \subset \Pi_N \Pi_{N+1} (C_{n_i}) $ It  is clear that $D_N \subset \Pi_N(D_{N+1})$.  
 Now we show that $\Pi_N(D_{N+1}) \subset D_N$ by contradiction. 
Suppose that $ \Pi_N(D_{N+1})$ contains a path $(b,\gamma, a) $ not in $D_N$.  
We have by hypothesis that there is a $\delta > 0 $ so that $(b,\gamma, a) $ is distance greater than $\delta $ from $D_N$.  We consider a sequence of paths 
$\gamma_{n_i}  \ \in \ C_{n_i} $ so that $\Pi_{N+1}\gamma_{n_i}$ converge to  $(b',\gamma' ,a' )$ such that $\Pi_N( (b',\gamma' ,a' ) = (b,\gamma,a)$. 
By condition (v) and (viii) $T_N( \gamma_{n_i} ) $ must be in time $[-N+\epsilon_N,N-\epsilon_N]  $ 
or the age $a(T_N) $ must be greater than $2^{-N}+ \epsilon_N$and  $ \gamma_{n_i}( T_N)$ must be in  spatial interval $[-N+\epsilon_N,N-\epsilon_N]  $.
From this and (ix), we see that $\Pi_N\gamma_{n_i }$ must converge to $(b, \gamma, a) $ and the desired contradiction is achieved.

To construct our limit set $C$ (which will clearly be in $K$), we need to find a collection of aged paths $C$ so that for each 
$N, \ \Pi_NC = D_N$.  Fix  $(\gamma_N, a_N) \ \in D_N$.  By the above paragraph we can find inductively $( \gamma_M,a_M) \in D_M \ \forall M > N$ so that $\Pi_N (( \gamma_M,a_M)) \ = \ ( \gamma_N,a_N)$.   So $\gamma_M $ and $ a_M $ are cadlag functions defined on intervals $[c_M, d_M]$ so that 
$$
\forall \ M > M' \ \geq N, \quad [c_{M'}, d_{M'}]   \ \subset \ [c_M, d_M] \mbox{ and } \gamma_M \vert _{[c_{M'}, d_{M'}]} = \gamma_{M'}, \ a_M \vert _{[c_{M'}, d_{M'}]} = a_{M'}.
$$
We also have that $d_M$ tends to infinity as $M $ tends to infinity but that (by condition (ii)), $\lim_{M -> \infty } c_M > -\infty $.
We define $\gamma $ on $(\lim_{M -> \infty } c_M, \infty ) $ by $\gamma(s) \ = \ \gamma_M(s) $ for any (and by the consistency all) $M$ with $s \ \in \ [c_M, d_M].$  Similarly for 
$a$.

We have that $(\gamma, a) $
has the desired property.  We take $C$ to be the totality of paths that can be obtained in this way (i.e. starting from some $N$ and taking a convergent sequence of aged paths.  It is clear $C$ is our limit.  
\end{proof}

\smallskip

For $N = 1, \, 2 , \dots $, we consider $\Gamma_N $ as the random collection of aged paths $\Pi_N (\mathcal{X})$.

Let us fix a space time square $S_{N} = [-N, N]^{2}.$

\medskip

\begin{prop} \label{goodbegin}
Given a path $(b,\gamma,a)$ that intersects $S_N$ while of age at least $\epsilon$, let 
$$
\beta = \beta(\gamma, \epsilon) = \inf \{t: (\gamma (t),t) \in S_N \mbox{ and } (\gamma (t),t) \mbox{ has age } \geq \epsilon\} \, .
$$ 
Given $N \ge 1$ and $1 > \sigma > 0$ then $\exists$ $\eta_1 > 0$ so that 
\begin{eqnarray*}
\lefteqn{P \big[ \exists \ \gamma \mbox{ that intersects } S_N \mbox{ while of age at least } \epsilon } \\
& & \qquad \mbox{ and has variation greater than } \sigma/100  \mbox{ on } [\beta, \beta + 3 \eta_1] \big] \le 10^{-6} \sigma^{2} \, .
\end{eqnarray*}
\end{prop}

\begin{proof} We consider paths $\gamma $ for which the $\beta$ as defined above lies in $ \left( i \frac{ \epsilon }{3}, ( i+1)  \frac{ \epsilon}{3} \right]$ for fixed $i$.   ( There are $\leq \frac{7N}{ \epsilon }$ such $i$'s  for $\epsilon$ small). We fix such an $i$.)

We locally augment notation and write  $\beta(\epsilon )$ as $\beta(\epsilon , \gamma )  )$ as we will discuss multiple paths.
We are interested in the paths behaviour after the time $\beta(\epsilon )$.  As such it is only necessary to treat a ``good" representative.  While it may be true that the evolution of the age of path $\gamma$ immediately before $\beta(\epsilon , \gamma )$ was very rapid due to several coalescences, for every $\gamma$, the behaviour of the path $\gamma$ on interval $[\beta(\gamma ), \infty )$ 
will equal that of a path $\gamma '$ on interval $[\beta(\gamma ), \infty ) \ = \ [\beta(\gamma ' ), \infty )$ for some path whose age at time $i \epsilon /3 $ is at least $\epsilon /2 $.  
As such to establish the proposition it is enough to treat $\gamma$ having this property.  Henceforth we drop the dependence of $\beta$ on $\gamma$ from the notation.



So we are interested in the behaviour on interval $[\beta,\beta+ 3 \eta_1] $ of paths $\gamma$ having the property that at time $i \epsilon /3 $ the path has age at least $\epsilon /2$ and such that in time interval $i \epsilon /3, (i+1) \epsilon /3 ] $ the path $\gamma$ meets spatial interval $[-N,N]$.  We first note that by Proposition \ref{density}, the density of the translation invariant collection of walkers of age at least $\epsilon /2 $ at time $i \epsilon /3 $ (or indeed any time) is equal to $K/(\epsilon /2 ) ^{ 1 / \alpha } $.  We now consider (for comparison purposes) the system of stable processes beginning with these walkers    evolving independently on time interval $[ i \epsilon /3 , (i+1) \epsilon /3] $ without coalescence.  This system at time $(i+1) \epsilon / 3 $ will have the same density ($K/(\epsilon /2 ) ^{ 1 / \alpha } $) and will again be translation invariant.  Thus the expected number of points for our comparison system in $[-N,N]$ at time $(i+1) \epsilon /3$ is exactly $2N K/(\epsilon /2 ) ^{ 1 / \alpha }$.  
By the Markov property (applied when a  process first enters $[- N, N]$ each process that touches $[-N, N]$ has a probability $\geq C_{N, \epsilon }  > 0$ of being within interval $[-N,N]$ at  time $\frac{(i+1) \epsilon }{3}$ where by symmetry  e.g. $C_{N, \epsilon }  > 1/3$ for $\epsilon $ small enough and $N $ large enough.  Thus provided $\epsilon$ was fixed sufficiently small, we have that the expectation of the number of comparison processes that touch spatial interval$[-N,N]$ in  time interval $[ i \epsilon /3 , (i+1) \epsilon /3] $ is bounded above by 
$$
6N K/(\epsilon /2 ) ^{ 1 / \alpha } \, .
$$
This bound must then also apply to our original system of coalescing paths by simple stochastic domination.

 So the expectation of the total number (i.e. for every relevant $i$) is thus $\leq \frac{7N}{ \epsilon} \frac{2kN}{(\frac{\epsilon}{2})^{\frac{1}{\alpha}}}.$ By the Markov property (for each stable process) if we choose $\eta_1$ so that 
$$
P(\sup_{s \leq 3\eta_1} |X(s)-X(0)|\geq  \sigma / 100 ) \leq \sigma^2/ \left( \frac{7N}{ \epsilon} \frac{2N}{(\frac{\epsilon}{2})^{\frac{1}{\alpha}}}.10^7 \right) \, ,
$$ 
we have that outside probability  $ \frac{\sigma^{2}}{10^{7}}$, the variation on $[\beta, \beta + 3 \eta_1]]$ is less than $\sigma / 100$ for all these paths. \end{proof} 

\medskip

\begin{prop} \label{goodcompact}
Given a path $(b, \gamma,a)$ that intersects $S_N,$ while of age at least $\epsilon , \ \gamma $, let $\beta  = \beta ( \gamma, \epsilon) \ = \ \inf \{t: (\gamma (t),t) \in \ S_N \mbox{ and } (\gamma (t),t) \mbox{ has age } \geq \ \epsilon\}$.
\noindent Given $N, \sigma \quad  \exists \ ( \delta (r)) _{r \geq 1}$ a positive sequence tending to zero as $r$ tends to infinity so that \\
$P [ \exists \ \gamma $that intersects $[-N, N]^{2}$ while of age at least $ \epsilon $ and so that for some $r \geq 1, \omega ( 2^{-r},\gamma ,[ \beta , N]) > \delta(r) ] \   \leq  \frac{\sigma^{2}}{10^{6}}$.
\end{prop}

\medskip

\begin{proof} The claim follows from the observations $ \omega ( 2^{-r},\gamma ,[\beta , N]) \to 0$ as $r \to \infty$ and the total number of paths is bounded in probability (from the previous proposition). 
\end{proof} 
\smallskip

{\it Remark}:  The results above speak to properties  (i) and (iv)  in the definition of compact set $K(  ( \epsilon_N)_{N \geq 1},  ( M_N)_{N \geq 1}, ( \delta_N)_{N \geq 1})$ while the next proposition addresses (iii). Property (ii) follows from lemma \ref{nobigage} 

\medskip

\begin{prop} \label{propnear}
Given $\eta_2 , \sigma  > 0$, there is $N'$ so that the probability that a path $(b,\gamma,a)$ satisfies
\begin{enumerate}
\item[(i)] $\gamma$ hits $S_{N}$ while of age at least $2^{-N} $,
\item[(ii)] there exists $t \in [-N,N]$ so that $a_\gamma(t) \geq \eta_2$ and $ \gamma(t) \in [-N', N']^c$,
\end{enumerate}
is bounded by $\sigma$.
\end{prop} 

\smallskip
 
\begin{proof}
\noindent Fix a positive integer $k$ and let $A_k$ be the event that for some $i$, there is a path $\gamma$ that intersects $I_k$ in time $\left[ \frac{i \eta_2}{3} , (i+1) \frac{\eta_2}{3} \right]$ so that
\begin{enumerate} 
\item[(i)]  $\left[ \frac{i \eta_2}{3} , (i+1) \frac{\eta_2}{3} \right] \cap [-N, N] \not= \emptyset$;
\item[(ii)] the path has age $\geq \frac{\eta_2}{2}$ at time $\frac{i \eta_2}{3}$;
\item[(iii)] $I_{k} = [2^{k} N^{\beta}, 2^{k+1} N^{\beta}] \cup  [-2^{k+1} N^{\beta}, -2^{k} N^{\beta}]$ for $\beta \ge \frac{2}{\alpha}$.
\end{enumerate}
Then we have since the number of such $i$ is bounded above by $\frac{7N}{\eta_2}$  and the expected number of such paths is dominated by $\frac{K 2^{k} N^{\beta}}{\eta_2^{\frac{1}{\alpha}}}$ then the expectation of this number is bounded above by 
$$
\frac{7N}{\eta_2}  \frac{K 2^{k} N^{\beta}}{\eta_2^{\frac{1}{\alpha}}} q_N \, .
$$
where $q_N$ is the probability that a stable process starting at some point in interval $[-N,N]$ visits $I_{k}$ before time $2N$ which is of order
$N/N^{\beta (\alpha +1)}2^{k(\alpha+1)}$. Thus for some $K' > 0$
$$
P(A_{k}) \le  \frac{7 K K'}{\eta_2^{1 + \frac{1}{\alpha}}} N^{2-\beta \alpha} 2^{-k\alpha} \, .
$$
Since $\beta \ge \frac{2}{\alpha}$, summing this over $k \geq k_0$ gives $P(\cup_{k \geq k_{0}} A_{k}) < \sigma$ for $k_{0}$ sufficiently large.  Since $\cup_{k \geq k_{0}} A_{k}$ contains the event in the statement we are done.
\end{proof}

\medskip

We use Proposition \ref{propnear} above to show:

\begin{lemma} \label{acont}
For every $N \in \IN$ and $\epsilon > 0$ there exists $\eta > 0$ so that 
$$
P\big( \exists \, (b,\gamma,a) \mbox{ so that } \Pi_N \gamma \ne \emptyset \mbox{ and so that } a(\cdot) \mbox{ jumps twice in an interval of length } 
\eta \big) 
$$
is less than $\epsilon$.
\end{lemma}

{\it Remark:}  Lemma \ref{acont} above addresses (vii) of the definition of the compact sets $K(\vartheta)$. 

\begin{proof}
We fix $N$ and $\epsilon$.  First pick $N'$ according to Proposition \ref{propnear} applied to $N$ with $\eta_2 = 2^{-N}/10 $ and $\sigma \ = \ \epsilon/ 100$.  Then pick $N''$ in this way with $N'$ substituted for $N$.  Pick $M$ so that the probability that the number of paths in $\Phi_{N+2} (C)$ to touch $(-N'',N'') \times (-N,N) $ is greater than $M$ is less than $\epsilon / 100)$.
Let the complements of these "expected events"  be denoted by $B_1, B_2, B_3$:\\

\indent
$B_1$ is the event that there exists a path which touches spatial interval $[-N,N]$ in time interval
$[-N-1,N+1]$ while having been outside spatial interval $[-N',N']$ while of age greater than $\eta_2 = 2^{-N}/10 $, \\
\indent
$B_2$ is the event that there exists a path which touches spatial interval $[-N',N']$ in time interval
$[-N-1,N+1]$ while having been outside spatial interval $[-N'',N'']$ while of age greater than $\eta_2 = 2^{-N}/10 $, \\
\indent
$B_3$ is the event that the number of paths to touch spatial interval $[-N'',N'']$ during temporal interval $[-N-1,N+1]$ while having age greater than $2^{-N-2}$ exeeds $M$. \\
We divide up the event in question into the union of events
$A(i,N)$ which is the event that a path of age $ \geq 2^{-N} $ meets two paths also of age $ \geq 2^{-N} $   in time interval
$[i \eta, (i+2) \eta]$ (which intersects $[-N,N]$) and all three paths were in $(-N'',N'')$ at time $i \eta - 2^{-N-1}$.  It is easy to see that 
the probability of $\cup _{i \eta \in (-N-1, N)} A(i,N) $ occurring but not one of the $B_i$ is less than 
$Const( N/ \eta ) 2^{2N/ \alpha }  M^3 \eta^{2/  \alpha } $: we simply note that for a fixed such interval, for the event to occur (and none of the $B_i$ ) at time $i \eta - 2^{-N} /4 $ all three of the processes must be among the at most $M$ processes of age at least $2^{-N} / 2 $ in spatial interval $[-N'',N'']$.  We have at most $M^3 $ choices for the three processes.  Uniformly over the positions of the three processes a time $i \eta - 2^{-N} /4 $, the probability that the first two meet and then the third in the given time interval is bounded by
$\eta^{2/ \alpha} 2^{2N/ \alpha } $.

 Given that our value $M$ has been fixed 
(and that $\alpha < 2$) this upper bound will be less than $\epsilon / 100$ if $\eta$ is chosen small enough.  Thus with these choices of $N', n'', M$ and $\eta $ the probability of the event occurring is less than 
$ \epsilon / 100 + \epsilon / 100 + \epsilon / 100 + \epsilon / 100  < \epsilon$ and the result follows.

\end{proof}

In a similar way we can show the following which is relevant to (vi) of the definition of $K(\vartheta)$.

\begin{prop} \label{reg1}
For each $\sigma, \epsilon > 0, \ N < \infty , \ \exists \eta > 0 $ so that 
$$
 P[  \exists \mbox{ a path that comes within } \eta \mbox { of  an older path while having age in } 
 $$
 $$
  ( \epsilon - \eta, \epsilon + \eta ) \mbox{ while in }  [-N,N]^2] \ < \ \sigma ^ 2/ 10^6 \, .
$$
\end{prop}
Again by Proposition \ref{propnear} we can restrict attention to paths within $N' $ of the origin. Again we bound the number of such paths that touch this spatial interval during time interval $(-N,N) $ while of age at least $2^{-N-1} $.  The argument is now as with Lemma \ref{acont}.

\begin{prop} \label{propagecont}
For each $\sigma, \epsilon > 0, \ N < \infty , \ \exists \epsilon_N > 0 $ so that for event
$A(N, \epsilon _N) \ \equiv  \  \{  \exists\gamma \mbox{ having age in interval }  (2^{-N} - \epsilon_N, 2^{-N}, + \epsilon_N)
\mbox{ at times in }( -N- \epsilon_N , -N+ \epsilon_N) \mbox{ and } (N- \epsilon_N, N+ \epsilon_N) \mbox{  while in interval spatial } [-N,N]\} $
satisfies
$$
	 P[  A(N, \epsilon _N) ] \ < \  \sigma ^ 2/ 10^6.
$$
\end{prop}
To see this for the time interval $(N- \epsilon_N, N+ \epsilon_N)$, we first choose $N'$ so that (using
Proposition \ref{propnear}), outside a set of probability $\sigma ^ 2/ 10^7.$ any path that meets
$(-N-1, N+1)$ in time interval $(-N-1, N+1)$ must be within $N'$ of the origin while having age at least $2^{-N}/3$.  Outside  this small probability event, the claimed event lies in the existence of a path
such that at time $N- 2^{-N-1}$ lies in $(-N',N')$ and has age in interval of length $4 \epsilon_N$
around $2^{-N-1}$.  Given Corollary \ref{denage}, we obtain the result.

We can equally address property (v) in our definition of compact $K$:
\begin{lemma} \label{lemrav}
For each $ N$ and each $\delta > 0 , \ \exists \ \epsilon_N > 0 $
so that  the probability that there exists a path $\gamma $ which first hits $[N- \epsilon_N, N + \epsilon_N] $ at a time before or equal to its first time of hitting $[-N,N] ^2$ 
while of age at least $2^{-N} $ is less than $\delta $.
\end{lemma}

\begin{proof}
We denote the ``bad" event whose probability we wish to bound by $B_N$.
It follows from the self similarity properties of the stable process that for a stable process $\{X(s) \}_ { s \geq 0 }$ starting at $1$ with $\tau \ = \ \inf \{s: X(s) \leq 0\}$,
we have $X(\tau ) < 0 $.  By quasi left continuity we get that 
$$
c( \epsilon ) \equiv \ P^1 ( \{X(s) \ 0 \leq s \leq \tau \} \cap [- \epsilon, \epsilon] \ne \emptyset ) \ \Rightarrow \ 0
$$
as $ \epsilon $ tends to zero.  So by scaling we have for $\tau $ now equal to $ \ \inf \{s: X(s)X(0) \leq 0\}$
$$
\sup _ {|x| \geq a } P^x ( \{X(s) \  0 \leq s \leq \tau \} \cap [- \epsilon, \epsilon] \ne \emptyset ) \ = \ c( \epsilon / a ) .
$$

We divide up $B_N$ into four parts: 

$B_N(1)$: a path enters $[-N,N]^2 $ while of age at least $2^{-N} $ which had been outside spatial interval $[-N',N'] $ while of age in  $[2^{-N} /3, 2^{-N}]$.

$B_N(2)$: the number of paths inside  $[-N',N'] \times [-N,N] $ of age at least $2^{-N} /3$ is greater than $N''$.

$B_N(3)$: there exists a path that achieves age $2^{-N} $ while spatially in  $[N- \epsilon'_N, N + \epsilon'_N] $.

$B_N(4)$: $B_N$ occurs through a path that hits  $[-N,N]^2 $ with age at least $2^{-N} $ which achieved age $2^{-N}$ while outside   $[N- \epsilon'  _N, N + \epsilon'_N]$. 

The variables $N', \ N'' $ and $\epsilon'_N$ will be specified as the proof progresses.

By Proposition \ref{propnear} if $N' $ is fixed high enough, then $P(B_N(1)) < \delta / 4 $.  Similarly we have that for $N''$ sufficiently large $P(B_N(2) \backslash B_N(1)) < \delta / 4 $.
By applying the Markov property at age time $2^{-N} / 3 $ we easily see that $P(B_N(3) \backslash (B_N(1) \cup B_N(2) ) ) < C N'' \epsilon' _N 2^{N/ \alpha } <  \delta / 4 $
if $\epsilon ' _N $ is fixed small enough.   Finally
$$
P(B_N(4) \backslash (B_N(1) \cup B_N(2) \cup B_N(3)  ) ) < C N'' c(\frac{ \epsilon _N }{ \epsilon ' _N}) 
$$
 which is less than $\delta /4 $ if $\epsilon_N $ is chosen small enough.
\end{proof}

Propositions \ref{goodcompact}, \ref{reg1}, \ref{propagecont}, Lemma \ref{acont} and Corollary \ref{corstab} as well as the proof of Proposition \ref{goodbegin} yield.
\begin{prop} \label{reg2}
For each $\sigma,$ $ \exists \vartheta \ = \  ( \epsilon_N)_{N \geq 1},  ( M_N)_{N \geq 1}, ( \delta_N)_{N \geq 1}$ so that 
$$
 P[  C \not\in K( \vartheta) ) ] \ < \ \sigma ^ 2/ 10^6 \, .
$$
\end{prop}
{\it Remark:} This shows that our measure on aged paths is tight given Proposition \ref{prok}.
\vspace{0.3cm}

A $\theta $  process is simply a stable process beginning at a some time point in $\theta \IZ^{2}$.  The $(\theta i , \theta j)$-process is simply the stable process beginning at 
position $ \theta i $ at time $ \theta j $ (so that $ \theta$ processes are simply the collection over $(i,j) \in \ \IZ^{2}$ of $(\theta i , \theta j )$-processes.

\medskip

\begin{prop} \label{propapprox}
\noindent Given $N, N', \eta, \epsilon \hspace{0.2cm} \exists \hspace{0.2cm} \theta > 0$ so that  $P[\exists$ a path so that $( \gamma (s),s) \in [-N', N'] \times [-N',N]$ while having age $\geq \eta$ but is not\''coalesced```with a $(\theta{i}, \theta{j})$ process by age $2{\eta}] < \epsilon$ for $(\theta{i}, \theta{j}) \in [-N', N'] \times [-N', N]$

\noindent (Note: we can define the "age" for a $ \theta $ process $X^{i \theta, j \theta }$ (considered among other $\theta $-processes)  at time $s > j \theta $ as the
largest value $s - j' \theta $ over $j'$ so that for some $ i', \ X^{i \theta, j \theta }_s \ = \ X^{i' \theta, j' \theta }_s$. 
\end{prop}

\smallskip

\begin{proof}
We simply consider processes that at a time $\frac{i \eta}{3}$ in $\left(\frac{\eta}{3}, \frac{2 \eta}{3}\right)$ have life for some $i$ with $\frac{i \eta}{3} \in  [-N, N]$. Fix (provisionally) an $i$.

\vspace{0.1cm}
\noindent For each one enumerated in some arbitrary manner we can apply the Markov property at the time $\tau$ the age becomes $\eta$ in the following at times $\tau, \tau + \theta, \tau + 2 \theta,...$ there is a chance bounded away from $0$ that the stable process will coalesce with the $\theta-$process nearest it, so the probability that the coalescence is  not effected by time age $2\eta$ is
 $\leq (1-C)^{[\frac{\eta}{\theta}]}$, so the expected number of non coalesced processes is bounded (as usual) by $\frac{7N}{\eta} \times \frac{KN'}{( \eta''  )^{1 /  \alpha } } \times (1-C)^{(\frac{\eta}{\theta})}< \sigma$
  for $\theta$ small. $C$ here is (by scaling)  the infimum over $| x | \leq 1 $ of the probability that two independent stable processes beginning at time 0 at $x$ and $0$ hit before time $1$.
\end{proof}

This begets
\begin{cor} \label{cordelta}
Given $N$ and $\sigma > 0 $, there exists $N'$ and $ \theta > 0 $ so that if $\L $ is the collection of coalescing stable processes in $\theta \IZ^2 \cap [-N',N'] \times [-N',N] $, then 
outside probability $\sigma $  for every path $\gamma $ with $(\gamma (s), s ) \in [-N,N]^2$ for some $s \in \ [-N,N]$ with $a(s) \geq 2^{-N}$, we have $\gamma ' \ \in \ \L $ with distance 
$$
\rho ( \Psi_N \circ \Phi_{2^{-N}} \gamma , \Psi_N \circ \Phi_{2^{-N}} \gamma ')  < \sigma / 10
$$
where by abuse of notation $\Psi_N \circ \Phi_{2^{-N}} \gamma '$ is the path with the operator $\Phi_{2^{-N}} $ applied for the age of $\gamma ' $ among $\theta $ processes. 
\end{cor}

\begin{proof}
Given $N $ and $\sigma $ let us apply Proposition \ref{goodbegin} with $\epsilon = 2^{-N} $ to obtain $\eta_1 $ satisfying the desired condition.  We can also, arguing as in Lemma \ref{acont}
suppose that $\eta_1$ is sufficiently small that no path hitting $[-N,N]^2$ while of age greater than $2^{-N}$ has a jump in $[a(2^{-N})- \eta_1, a(2^{-N}) + \eta_1]$ outside this probability.
Now given this $\eta_1$ (which we can take to be small compared to $2~^{-N} $, let $ \eta_2  $ be less than $ \sigma \eta_1 / 100 $.   We apply Proposition \ref{propnear} with $\eta = \eta_2 $ and $ \sigma $ equal to our fixed $\sigma^2 / 10^6$.  This yields our desired $N'$ (We here also suppose that outside this probability no path hitting $[-N,N]^2 $ has age greater than $N'$.).  
Applying Proposition \ref{propapprox}, with $N, N', \eta $ and $\epsilon = \sigma^2 / 10^6 $ we have our $\theta$ and outside of probability $2 \sigma^2 / 10^6$, every path $\gamma $ as above has coalesced with a path in $\L $ before it has age $2 \eta $.  The result now follows from Proposition \ref{goodbegin}.
\end{proof}
\noindent Propositions \ref{goodbegin} , \ref{propnear} and \ref{propapprox} and Corollary \ref{cordelta} yield

\begin{prop} \label{propfinal}
\noindent  \hspace{0.2cm}$ \forall \sigma > 0$ there exists $N, N', \theta$ so that outside probability $\sigma$ the $\rho $ distance between the paths resulting from the stable web and the paths resulting from the $\theta \IZ^2 \cap [-N',N'] \times [-N',N+1] $ beginning after  age $2^{-N} $ (with age as mentioned after Proposition \ref{propapprox} )  when the processes touch  $  [-N, N]$ is less than $\sigma $.
\end{prop} 
 
\medskip

We denote the system of $\theta$-processes by $\mathcal{X}^\theta$ and the system of $\theta$-processes of age at least $\delta$ by $\mathcal{X}^\theta_\delta$. We claim that $\mathcal{X}^\theta$ plays the role of an skeleton for the stable web in an analogous way to the definition of an skeleton for the Brownian Web, see \cite{FINR}.

\section{\Large\bf Convergence in Distribution}

\noindent Consider a random walk $(W_n)_{n\ge 1}$ such that $W_n = \sum_{i=1}^n Z_{n}$ with $(Z_n)_{n=1}^\infty$ iid random variables whose distribution is in the domain of attraction of a stable symmetric $\alpha \in (0,2)$ r. v. $X = X^{0,0}_1$ ($X^{0,0}$ defined as in Section \ref{sec:agedproc}). Let $p(x) = P(Z_{1} = x)$, $x \in \mathbb{Z}$ be its transition probability function. Since the best convergence result is not our focus we assume that $p(\cdot)$ is symmetric and satisfies 
$$
x^{1 + \alpha} p(x) \rightarrow C \in (0, \infty) \, , \quad \mbox{as } |x| \rightarrow \infty \, , 
$$  
where the $C$ is chosen to be compatible with $X$.  Thus we have that
$$
\left( \frac{W_{\lfloor nt \rfloor}}{n^{\frac{1}{\alpha}}}  \right)_{t \geq 0} 
$$
converges in distribution under the Skorohod topology to $(X^{0,0}_{t})_{t\geq 0}$. We have the Gnedenko local CLT, see \cite{GK},
$$ 
\sup_{\vert x_{0} \vert \leq K \eta^{\frac{1}{\alpha}}}\left \vert n^{\frac{1}{\alpha}} P\left( W_n = x_{0} \right) - P_{ X} \left( \frac{x_{0}}{n^{\frac{1}{\alpha}}}  \right) \right \vert \rightarrow 0 \, , \quad \mbox{as } n \rightarrow \infty \, . 
$$
This is enough to show convergence of the Green functions:
$$
n^{(\frac{1}{\alpha}-1)} \sum^{\infty}_{j=1} P( W_j = [u n^{\frac{1}{\alpha}}]) e^{-\beta \frac{j}{n}} \to G_{\beta} ( u) \, , \ \ \forall \hspace{0.3cm} \beta > 0 \, ,
$$ 
uniformly on compact intervals for $u$.

\medskip

\noindent From this and standard optional stopping, we have

\medskip

\begin{lemma}
For every $\beta > 0$, if $W^{1}, \  W^{2}$ are two iid random walks with increments distributed as $p(\cdot)$ and starting at $W^{1}_{0}= 0$ and $W^{2}_{0} = [n^{\frac{1}{\alpha}}u]$ and $T_n = \frac{1}{n} \inf \{j: W^{1}_{j} = W^{2}_{j}\}$ then $\lim_{n\rightarrow \infty} E e^{- \beta T_n} = E e^{- \beta T}$ where $T = \inf \{t: X^{1}_{t} = X^{2}_{t} \}$ for $X^{1},$ $X^{2}$ iid stable processes distributed as $X^{0,0}$ starting at $X^{1}_{0} = 0$ and $X^{2}_{u} = u$. Therefore $T_n \Rightarrow T$.
\end{lemma}

\medskip

\noindent From this we obtain

\begin{lemma} \label{lemdelta}
For $N$, $N^{1}$, $\epsilon$ and $\theta$ fixed positive and finite, the system of coalescing stable processes starting from points in $\theta \IZ^{2} \cap [-N^{1}, N^{1}] \times [N, N]$ is the limit of the system of coalescing random walks starting on $[-N^{1} n^{\frac{1}{\alpha}}, N^{1}n^{\frac{1}{\alpha}}] \times [-Nn, Nn]$ starting from points in $(\theta n^{\frac{1}{\alpha}}\IZ \times \theta n \IZ)$ and appropriately rescaled.
\end{lemma}

\medskip


Recall the definitions of Section \ref{sec:agedproc} and Proposition \ref{density}. The system $\overline{X}$ of $\alpha$-stable processes starting from full occupancy at time $0$ is scale invariant and in particular the density scales as $k/t^{\frac{1}{\alpha}}$ for some constant $k$ not depending on $t$.  We now note that the density for coalescing random walks scales (when suitably renormalized) in the same way.

\smallskip

\begin{prop} \label{prop:densrw}
For coalescing random walks on $\IZ \times \IR_{+}$ beginning with full occupancy the density at time $n \approx k/n^{\frac{1}{\alpha}}$
where $k$ is the constant for the continuous time coalescing processes obtained in Proposition \ref{density}.
\end{prop}

\begin{proof}
For the system of $\alpha$-stable coalescing processes the density at time $t$ is the (increasing) limit as $\theta \downarrow 0$ of the processes beginning at $\theta \IZ$. Thus for every $\epsilon > 0$, there exists $\theta > 0$ so that the density of coalescing $\theta$-processes at given times $t_{1}, t_{2}$ are greater than $(k - \epsilon)/t_{1}^{\frac{1}{\alpha}}$ and  $(k - \epsilon)/t_{2}^{\frac{1}{\alpha}}$.

We first take $t_{1} = 1$. Now for $\theta$ as above, we consider the coalescing random walks beginning at $\theta n^{\frac{1}{\alpha}} \IZ$. By the invariance principle and following discussion we get that for $n$ large the density of these coalescing random walks is at least  $(k - 2 \epsilon)/n^{\frac{1}{\alpha}}$ for $n$ large. Hence by monotonicity it is at least $(k - 2 \epsilon)/n^{\frac{1}{\alpha}}$ for the full process of coalescing random walks (i.e. starting from full occupancy).

On the other hand we can via Bramson and Griffeath arguments \cite{BG} show that there exists $m < \infty$ so that $\forall \hspace{0.1cm} n$ the density of coalescing random walks at time $T$ is bounded above by $m/n^{\frac{1}{\alpha}}$. In particular at $T = r n$, $r$ small, the density is bounded above by $m/(r^{\frac{1}{\alpha}} n^{\frac{1}{\alpha}})$. We now couple this to a coalescing system of random walks starting with occupancy at $\theta n^{\frac{1}{\alpha}} \IZ$ at time $T = r n$ and we take $t_{2} = (1-r)$ then the density of the full process of coalescing random walks at time $n = r n + (1-r) n$ is equal to the density of the $\theta n^{\frac{1}{\alpha}}\IZ$ random walks at time $(1- r) n$ plus the density of walks of the full process that have not coupled with a $\theta$ random walk by time $n.$ But (if $\theta$ is sufficiently small) this latter density will be smaller than $\epsilon$ while the former density (by invariance) will be less than $(k + \epsilon)/((1- l)^{\frac{1}{\alpha}} n^{\frac{1}{\alpha}})$.  Thus the density of full random walks at time $n$ will be bounded above by 
$$
\frac{k + \epsilon}{(1 - l)^{\frac{1}{\alpha}} n^{\frac{1}{\alpha}}} + \epsilon \, .
$$
We now let $r \downarrow 0$ and then $ \epsilon \downarrow 0$ to obtain our result.  
\end{proof}

\medskip

From now on we work with systems of continuous time coalescing random walks starting at points on $n^{-\frac{1}{\alpha}}\IZ \times n^{-1} \IZ$, making jumps from $in^{- 1/ \alpha } $ to $jn^{- 1/ \alpha }$ at rate  $np(j-i)$. The age of such a random walk is defined to be continuously increasing at rate 1 on intervals of no coalescence and when a coalescence occurs the age jumps to the age of the older path. We call it an aged random walk.  We denote the collection of aged random walks (suitably renormalized) as $\mathcal{W}_n$ for $n\ge 1$ (so $n$ is the scaling parameter). Given Proposition \ref{prop:densrw} we can easily prove the following analogue of Proposition 4:

\smallskip

\begin{prop}
Given an aged random walk path $(\gamma,a) \in \mathcal{W}_n$ that intersects $S_N$, while of age at least $\epsilon$, let 
$$
\beta = \beta( \gamma, \epsilon) \ = \ \inf \{t: (\gamma (t),t) \in \ S_N \mbox{ and } a(t) \geq \epsilon\}.
$$
Given $N$ and $\sigma$ there exists $\eta_1$ so that 
$$
P \big[ \exists \ (\gamma,a) \in \mathcal{W}_n \mbox{ that intersects } S_N \mbox{ while of age at least } \epsilon \mbox{ and has variation }
$$
$$
\mbox{ greater  than } \sigma \mbox{ on } [\beta, \beta + 3 \eta_1] \big] \leq  \frac{\sigma^{2}}{10^{6}} \, ,
$$
for every $n$.
\end{prop}

\medskip

We similarly have analogues of Propositions \ref{goodbegin} ,\ref{goodcompact}, \ref{reg1} \ref{reg2}, \ref{propnear}, \ref{propapprox} and \ref{propfinal} as well as Lemmas \ref{acont} and \ref{lemrav}. This yields:

\smallskip

\begin{prop} \label{reg2}
For each $\sigma > 0$, there exists $\vartheta = (( \epsilon_N)_{N \geq 1},  ( M_N)_{N \geq 1}, (\delta_N)_{N \geq 1})$ so that for each $\delta > 0, $
there exists $n_0  \ = \ n_0( \delta ) \ < \infty$ so that for $n \geq n_0$
$$
 P[  \mathcal{W}_n \not\in  K^ \delta  ( \vartheta) ] < \sigma^2/ 10^6 \, .
$$
In particular $(\mathcal{W}_n)_{n\ge 1}$ is a tight family of random elements of $H$.
\end{prop}

{\it Remark:} We need to consider $K^\delta$ rather than simply $K$ since the convergence of renormalized random walks to continuous time stable processes ensures the desired convergence for paths $\Pi_N \gamma $ for $n$ large enough (depending on $N$).

\medskip

The aim from this point is to prove weak convergence of $\mathcal{W}_n$ to $\mathcal{X}$. Our argument uses the approximation of $\mathcal{X}$ by $\theta$-processes and we need an analogous approximation for the system of aged random walks. So we consider the system of $\theta$-random walks associated to the scaling parameter $n$ as the collection of rescaled coalescing random walks starting at $\theta \mathbb{Z}^2$ that evolves as before, i.e. making jumps from $in^{- 1/ \alpha } $ to $jn^{- 1/ \alpha }$ at rate  $np(j-i)$.

Given a collection of $\theta$-random walks, we can, just as in the original process, speak of ages of paths: the age of a path $X^{\theta i , \theta  j}$
at time $s> j \theta$ is simply
$$
s - \inf \{ \theta l : \exists k: X^{\theta i , \theta  j}_s = X^{\theta k , \theta  l}_s \} \, .
$$
We then denote by $\mathcal{W}^\theta_n$ the system of aged coalescing $\theta$-random walks with scaling parameter $n$ and given $N$ and $N'$ we write $\mathcal{W}^{\theta,N'}_n$ for the system of coalescing $\theta$-random walks with scaling parameter $n$ beginning at points $\theta \IZ^2 \cap [-N-1, N] \times [-N',N']$. (Usually $N$ is given and so is dropped from the notation.)

\smallskip

As in Proposition \ref{reg2}, we can show that $(\mathcal{W}^\theta_n)_{n\ge 1}$ is a tight family of random elements of $H$. Moreover, we can prove as in Corollary \ref{cordelta} the following result:

\medskip

\begin{prop}
\label{pfdd}
Given $N, \ N' $ and $\theta$, the coalescing renormalized systems $\mathcal{W}^{\theta, N'}_n$ converge in distribution to $\mathcal{X}^{\theta, N'}$ as $n$ tends to infinity. Furthermore $ \Psi_N \circ \Phi_{2^{-N}} (\mathcal{W}^{\theta, N'} _n )$ converges in distribution to  $\Psi_N \circ \Phi_{2^{-N}} (\mathcal{X}^{\theta, N'} )$.
\end{prop}

We are now ready to establish weak convergence of $\mathcal{W}_n$ to $\mathcal{X}$.

To establish weak convergence it is sufficient to show that for a bounded and continuous $F$ on our space
$$
E[F(\mathcal{W}_n)] \  \rightarrow \ E[F(\mathcal{X})]
$$
as $n$ tends to infinity.  

We fix $\epsilon > 0$ and below for a set $K \in H$ we denote $K^\eta = \{\psi : \rho(\psi,K) \le \eta \}$. 

Now we fix a bounded continuous function $F$ on the set of aged path collections.
Given $\epsilon > 0$, we fix a compact set $K$ of collections of paths as in Proposition \ref{prok} so that the probability that $\mathcal{X} \in K$ is at least $1 - \epsilon / 3 $.  By the compactness of $K$ we have that there exists $\eta' > 0 $ so that
$$
\forall \psi \in K, \quad \sup_{\psi ' : \rho( \psi, \psi' ) < 100 \eta'} \vert  F(\psi ) - F(\psi')  \vert \quad < \ \epsilon / 100
$$ 
which immediately implies that 
$$
\forall \psi \in K^{\eta/50}, \quad \sup_{\psi ' : \rho( \psi, \psi' ) < 50 \eta'} \vert  F(\psi ) - F(\psi')  \vert \quad < \ \epsilon / 50.
$$

By Proposition \ref{pfdd} for any $\theta > 0 $ (and $N, N'$), 
$$
 \big| E [F(\mathcal{W}^{\theta}_n)] -E[F(\mathcal{X}^ \theta )] \big| \ \rightarrow \ 0.
$$
We choose $\theta$, $N$ and $N'$ according to Proposition \ref{propfinal} so that the distance between $\mathcal{X}$ and $\mathcal{X}^\theta$ is less than $\eta'$ outside probability $\epsilon / 10 $.
We can take this $\theta$, $N$ and $N'$ so that we have equally for each $n$, the distance between $\mathcal{W}_n $ and $\mathcal{W}_n^\theta$ is less than $\eta'$ outside this probability. We now have
\begin{eqnarray*}
 \big| E [F(\mathcal{X})] - E[F(\mathcal{X}^\theta)] \big| & \leq & E \big[ \big| F(\mathcal{X})] -F(\mathcal{X}^\theta) \big| \big] \\
 & \leq &  E \big[ \big| F(\mathcal{X}) -F(\mathcal{X}^\theta) \big| I_{\mathcal{X} \notin K} \big] \\ 
& & \qquad  + \, E \big[ \big| F(\mathcal{X})-F(\mathcal{X}^ \theta ) \big| I _{\rho ( \mathcal{X} , \mathcal{X}^ \theta) < \eta} I _{\mathcal{X} \in K}  \big] \\
& & \qquad \qquad + \, E \big[ \big| F(\mathcal{X}) -F(\mathcal{X}^\theta ) \big| I _{\rho ( \mathcal{X} , \mathcal{X}^\theta )  \geq \eta }  \big] \, .
 \end{eqnarray*}
This latter sum is bounded by $2 \epsilon ||F||_ \infty \  +  \  \epsilon \ + 2 \epsilon  ||F||_ \infty \le 5 \epsilon \, (1 \vee ||F||_ \infty)$ by Proposition \ref{propfinal}.

We also have that for any $\eta > 0$, $P( \mathcal{W}_n \in K^ \eta ) > 1 - \epsilon / 2$ for $n$ large and so we can argue as above that for universal $C$
$$
 \big| E [F(\mathcal{W}_n)] -E[F(\mathcal{W}_n^\theta)] \big| \ \leq C \epsilon \, (1 \vee ||F||_ \infty)
$$
for $n$ large which gives that $ \big| E [F(\mathcal{W}_n)] -E[F(\mathcal{X})] \big| \ \leq C' \epsilon \, (1 \vee ||F||_ \infty) $ for $n$ large.

\noindent {\bf Acknowledgements:}
We  would like to thank NYU-Shanghai and EPFL, Lausanne for hospitality .





\end{document}